\documentclass{amsart}

\usepackage{amsmath,amsthm,amssymb,IMjournal}
\usepackage{mathrsfs}
\usepackage{calligra}
\usepackage{verbatim}
\usepackage{xspace}
\usepackage{mathtools}
\usepackage{xcolor}
\usepackage{amsfonts}
\usepackage[none]{hyphenat}
\usepackage[colorlinks=black,urlcolor=black]{hyperref}
\usepackage{array}
\usepackage[usestackEOL]{stackengine}
\usepackage[nameinlink]{cleveref}
\usepackage{scalerel}

\newcommand{\f}{\mathbbmsl{F}}
\newcommand{\ie}{\textit{i.e.\@\xspace, }}
\newcommand{\ti}{\times}

\newcommand{\wh}[1]{\widehat{#1}}
\newcommand{\ra}{\rightarrow}
\newcommand{\ann}{\mathrm{Ann}\hspace{.01em}}
\newcommand{\w}{\mathbb{Z^{\geqslant \text{0}}}}
\newcommand{\ord}{\mathrm{Ord}\hspace{.01em}}
\newcommand{\ai}{\textit{\textbf{If}}\kern.5em}
\newcommand{\krn}{\mathrm{Krn}\hspace{.01em}}
\newcommand{\cq}{\coloneqq}
\newcommand{\anh}{\mathrm{Anh}\hspace{.01em}}
\newcommand{\cw}[1]{\textcolor{white}{#1}}
\newcommand{\dt}{\dots}
\newcommand{\dps}{\displaystyle}
\newcommand{\e}{\epsilon}
\newcommand{\so}{\stackon}
\newcommand{\su}{\stackunder}
\newcommand{\sz}{\scriptsize}
\newcommand{\stack}[3]{\so[8pt]{\su[8pt]{#1}{\cw{\sz{$#2$}}}}{\cw{\sz{$#3$}}}}
\newcommand{\g}{\gamma}
\newcommand{\op}{\oplus}

\renewcommand{\u}{\mathscr{U}}
\renewcommand{\a}{\alpha}
\renewcommand{\sup}{\mathrm{Spr}\hspace{.01em}}
\renewcommand{\dim}{\mathrm{Dmn}\hspace{.01em}}
\renewcommand{\c}{\mathscr{C}}
\renewcommand{\b}{\beta}
\renewcommand{\d}{\Delta}

\DeclareMathAlphabet{\mathbbmsl}{U}{bbm}{m}{sl}
\DeclareSymbolFont{letters}{OML}{ztmcm}{m}{it}
\DeclareMathAlphabet{\mathcalligra}{T1}{calligra}{m}{n}
\DeclareFontShape{T1}{calligra}{m}{n}{<->s*[2.2]callig15}{}

\numberwithin{equation}{section}

\newtheorem{The}[equation]{Theorem}
\newtheorem{Dfn}[equation]{Definition}
\newtheorem{Lem}[equation]{Lemma}
\newtheorem{Prs}[equation]{Proposition}
\newtheorem{Crl}[equation]{Corollary}
\newtheorem{Rmk}[equation]{Remark}

\begin{document}

\title{{\larger[3]U}nit {\larger[3]G}roups of {\larger[3]S}ome {\larger[3]G}roup {\larger[3]R}ings}

\author{\|Ali |Ashja'|, Tehran
		}

\rec {April 6, 2020}

\vspace{-28pt}

{\centering
\begin{dedicatory}
	CCordially Dedicated To The Reader \\
	May You Find Happiness \bigskip
\end{dedicatory}
\par}

\abstract 
		Let $RG$ be the gruop ring of the group $G$ over ring $R$ and $\u(RG)$ be its unit group.
	Finding the structure of the unit group of a finite group ring is an old topic in ring theory.
	In, G. Tang et al: Unit Groups of Group Algebras of Some Small Groups. Czech. Math. J. 64 (2014),
	149--157, the structure of the unit group of the group ring of the non abelian group $G$ with order
	21 over any finite field of characteristic 3 was established. In this paper,
	we are going to generalize their result to any non abelian group $G=T_{3m}$,
	where $T_{3m} = \langle x,y\,|\,x^m=y^3=1,\,x^y=x^t\rangle$.
\endabstract



\thanks
   The basis of this research has been extracted from dissertation.
\endthanks

\section{\bf{\larger[1]Introduction}}\label{sec1}

	Let $RG$ be the group ring of group $G$ over field $F$ and $\u(FG)$ be its unit group, \ie multiplicative subgroup containing all invertible elements. The study of unit group is one of the classical topics in ring theory that started in 1940 with a famous paper written by G. Higman \cite{Higman1940}. In recent years many new results have been achived; However, only few group rings have been computed. Unit groups are useful, for instance in the investigation of Lie properties of group rings (for example see \cite{Bovdi1999}) and isomorphism problem (for example see \cite{Creedon2008}).

	Up to now, the structure of unit groups of some group rings has been found. For instance, on integral group ring \cite{Jespers1993}, on permutation group ring \cite{Sharma2007FA4}, on commutative group ring \cite{Nezhmetdinov2010}, on linear group ring \cite{Maheshwari2016}, on quaternion group ring \cite{Creedon2009}, on modular group ring \cite{Raza2013} and on pauli group ring \cite{Gildea2010pauli}. In \cite{Davis2014}, the authors proved which groups can be unit groups, moreover, on properties of unit elements themselves instead of their groups structure \cite{Bakshi2017}.

	In this paper we will study the structure of unit group of group ring $\f_{3^n} T_{3m}$,
where ${T_{3m} = \langle x,y\,|\,x^m=y^3=1,\,x^y=x^t\rangle}$. Till now some cases, in characteristic 3, have been studied. For instance, in \cite{CreedonGildea2008}, the authors obtained the structure of unit group of $\f_{3^k} D_6$, in \cite{Gildea2010C3D6}, Gildea determined the structure of unit group $\f_{3^k} (C_3 \ti D_6)$ and in \cite{GildeaMonaghan2011} Gildea and Monaghan study groups of order 12 and recently in \cite{Monaghan2012}, Monaghan study groups of order 24. In this paper we generalize the result of \cite{Tang2014} on non abelian group $G$ of order $21$ to any group $G=T_{3m}$.

\section{\bf{\larger[1]Notations and Definitions}}\label{sec2}

	In this section, we bring some notations and lemma which we need for the proof of our main results. We denote the order of an element $g$ in the group $G$ by $Ord_G(g)$, the sum of all elements of subset $X$ in ring $R$ by $\wh{X}$, \ie $\sum_{r \in X}r$. Notice there is no need for $X$ to be a subring or subgroup, it defines for any arbitrary subset. In group ring $RG$, when $X$ is the subset of all different powers of $g$ (an element of group $G$), we may simply write $\wh{g}$ instead of $\wh{X}$. Also when $X$ is the right coset $\langle g \rangle h$, we may write $\wh{g}h$ for $\wh{X}$. In group, $x^y$ be the conjugate of $x$ by $y$ that is $x^y = y^{-1}xy$. Let $f:X \ra Y$ be an arbitrary function, then $Supp_X (f) = \{x\in X\, | f(x)\neq 0\}$. Also, we use the following notations: $\ann_R(a) = \{ r \in R \, | \, ra=ar=0 \}$, we denote a finite field of characteristic $p$ with order $p^n$ by $\f_{p^n}$.
If $E$ is a vector space over $F$, then $Dim_F(E)$ is the dimension of $E$ over $F$. Let $\u(R)$ be the unit group of ring $R$, \ie $\u(R) = \{u \in R \, | \, u^{-1} \in R \}$ and $J(R)$ be the jacobson radical of ring $R$. Now we state a useful definition.

\begin{Dfn}
		Let $RG$ be group ring of ring $R$ over the group $G$,
	$p$ be a prime number and $S_p$ be subset of all $p \! - \! elements$
	including identity element of $G$, \ie
	$S_p = \{ g \in G \, | \, \exists {n \in \w}; \, \ord_G(g)=p^n \, \}$.
	We define a binary map $T: G \ra R$ as follows:
	$$ T(g) = \left\{
				\begin{array}{*3c}
					1 & \ai & g \in    S_p \\
					0 & \ai & g \notin S_p \\
				\end{array}
			  \right.
	$$
		
		As we know that $T$ on $G$ is the base of $RG$, so we can linearly extend it to whole $RG$,
	of course no more remains binary. Also if see elements of $RG$ as functions from $G$ to $R$,
	that map every group element $(g)$ to its coefficient $(r_g)$, then their supports will be feasible.
	Now we can define ${\krn(T) \cq \{\a \in RG \, | \, \forall g \in G; \, \a g \in \ker_{RG}(T) \}}$
	and ${\sup(\a) \cq \mathrm{Supp}_G(\a)}$. Also ${\anh(a) \cq \ann_{RG}(a)}$ and
	${\dim(S) \cq \mathrm{Dim}_F(S)}$. \\
\end{Dfn}

\begin{Lem}
	\label{lemma}
		Let $F$ be a finite field of characteristic $p$, $G$ be a finite group,
	$T$ be a function defined as above and $s=\wh{S}_p$. Then: 
	\\ $\setlength\arraycolsep{1.5pt}
	\begin{array}{*4c}
		\quad (1) \quad & \quad J(FG)   & \subseteq & \krn(T). \\
		\quad (2) \quad & \quad \krn(T) & =         & \anh(s). \\
		\quad (3) \quad & \quad J(FG)   & \subseteq & \anh(s). \\
	\end{array}$
\end{Lem}

\begin{proof}
	\cite[Lemma 2.2 on p.~151]{Tang2014}. \\
\end{proof}

	In the next section we prepare some preliminaries that will be needed in our main theorem proof.

\section{\bf{\larger[1]Preliminary Results}}\label{sec3}

		Let $p = 3$, $G=T_{3m}$ group of order $3m$ that $t^3 \overset{m}{\equiv} 1$, $m=3k+1$
	and $(m,t-1)=1$. Also $s$ be as defined in \Cref{lemma},
	$\langle x \rangle$ be the cyclic subgroup generated by $x$ and
	$\langle x \rangle y$ be right coset of $\langle x \rangle$ with respect to $y$ that is
	$\langle x \rangle y = \{x^iy \, | \, 0 \leqslant i \leqslant m-1 \}$.
		By which defined $\wh{x}$ is sum of all different powers of $x$, so we have:
			$$\wh{x}  = \sum_{i=0}^{3k} x^i \qquad \qquad \wh{x}y = \sum_{i=0}^{3k} x^i y$$
			
		First we compute the conjugacy classes of $T_{3m}$:

\begin{Prs}
\label{conjugacy 39}
		Conjugacy classes of $T_{3m}$ are as below:
	\begin{equation*}
	\begin{array}{*4l}
		\c_0 & = & \{ 1							  \} &											  \\
		\c_i & = & \{ x^{j_i},x^{j_it},x^{j_it^2} \} & \text{k Conjugacy Classes with 3 Elements} \\
		\c_+ & = & \, \langle x \rangle y			 &											  \\
		\c_- & = & \, \langle x \rangle y^{-1}		 &											  \\
	\end{array}
	\end{equation*}
\end{Prs}

\begin{proof}
		It is clear that $T_{3m}$ has three types of elements:
	$x^i y^{-1}$, $x^i y$ and $x^i$. Conjugating these types with each other has 9 possibilities.
	By checking all of them we find that each elements can only be conjugate
	of some elements of same type. Since $x^y = x^t$ itself shows $y^{-1}x = x^ty^{-1}$,
	and with $t^3 \overset{m}{\equiv} 1$ also shows $yx = x^{t^2}y$, so we have:
	\begin{equation*}
	\begin{alignedat}{9}
		& (x^i y^{-1} && )^x = x^{-1}(x^i y^{-1} && )x = x^{-1}x^i(y^{-1} && x) = x^{-1}x^i(x^t
					  && y^{-1} && ) = (x^{-1}x^ix^t	 && )y^{-1} && = x^{i+(t^{\cw{1}}-1)} && y^{-1}\\
		& (x^i y	  && )^x = x^{-1}(x^i y		 && )x = x^{-1}x^i(y	  && x) = x^{-1}x^i(x^{t^2}
					  && y		&& ) = (x^{-1}x^ix^{t^2} && )y		&& = x^{i+(t^2-1)}		  && y	   \\
	\end{alignedat}
	\end{equation*}
	
		We know $(m,t-1)=1$, and with $t^3 \overset{m}{\equiv} 1$
	also we can conclude that $(m,t^2-1)=1$, so repeated conjugating with $x$
	implies that all elements of type $x^iy^{-1}$ form a conjagacy class,
	similarly all elements of type $x^iy$ form another one.	For elements of form $x^i$, notice that:
		$$(x^i)^{(x^jy^{\pm 1})} = (x^i)^{(y^{\pm 1})}$$ \\
		
		So it suffices to conjugate only by $y^{\pm 1}$.
	Now $x^y = x^t$ and $x^{y^{-1}} = x^{t^2}$ imply that they are as was claimed, Therefore:
	\begin{equation*}
	\begin{array}{*4l}
		\c_0 & = & \{ 1							  \} &											  \\
		\c_i & = & \{ x^{j_i},x^{j_it},x^{j_it^2} \} & \text{k Conjugacy Classes with 3 Elements} \\
		\c_+ & = & \, \langle x \rangle y			 &											  \\
		\c_- & = & \, \langle x \rangle y^{-1}		 &											  \\
	\end{array}
	\end{equation*}
\end{proof}
		Now we find the structure of annihilator:
\begin{Prs}
\label{g hat 39}
		Let $p = 3$, $G = T_{3m}$ and $s$ be as was defined in \Cref{lemma}.
	Then the structure of annihilator of $s$ will be as follows:
			$$\anh(s) = \{ a^-\wh{x}y^{-1} + a\wh{x} + a^+\wh{x}y \ | \ a^-+a+a^+=0 \}$$
\end{Prs}

\begin{proof}
		Since by \Cref{conjugacy 39} conjugacy classes of $G$ are known,
	It's easy to see that $G$ has three kinds of elements: Identity,
	elements of form $x^iy^{\pm 1}$ with order $3$ (they are conjugates of $y^{\pm 1}$) and
	elements of form $x^i \neq 1$ with order prime to 3 ($m=3k+1$).
	So $S_3 = \c_- \cup \c_0 \cup \c_+$,
	therefore, $\wh{S}_3 = \wh{\c}_- + \wh{\c}_0 + \wh{\c}_+ = \wh{x} y^{-1} + 1 + \wh{x} y$,
	sum of $3 \! - \! elements$	including identity.
	Let ${\a = \sum_i \a_i \in \anh(s)}$
	where, ${\sup(\a_i) \subseteq \c_i}$ and $s=\wh{S}_3$. Then we have:
	\begin{equation}
	\label{alpha s}
	\begin{split}	
		0 = \a . s
		& = (\sum_i \a_i)(\wh{x}y^{-1}+1+\wh{x}y) =					\\
		& = (\a_-+(\sum_{i=0}^k \a_i)+\a_+)(\wh{x}y^{-1}+1+\wh{x}y) \\
		& = (\a_- + \a_+\wh{x}y + (\sum_{i=0}^k \a_i)\wh{x}y^{-1} ) \\
		& + ((\sum_{i=0}^k \a_i) + \a_-\wh{x}y + \a_+\wh{x}y^{-1} )	\\
		& + (\a_+ + \a_-\wh{x}y^{-1} + (\sum_{i=0}^k \a_i)\wh{x}y )	\\
	\end{split}
	\end{equation}
	\par \medskip
	
		Notice that although $x$ doesn't commute with $y$, but $\wh{x}$ does, so for every $j$:
	\begin{equation}
	\label{x hat y 39}
	\begin{alignedat}{5}
		& x^j y		 && . \wh{x} y		&& = x^j   && . \wh{x} y^{-1} && = \wh{x} y^{-1} \\
		& x^j y^{-1} && . \wh{x} y		&& = x^j y && . \wh{x} y^{-1} && = \wh{x}		 \\
		& x^j y^{-1} && . \wh{x} y^{-1} && = x^j   && . \wh{x} y	  && = \wh{x} y		 \\
	\end{alignedat}
	\end{equation}

		So the conjugacy classes of three last parentheses of \Cref{alpha s} are different and
	since left side is zero, therefore, every parentheses should be zero separately. Hence:
	\begin{equation*}
	\begin{array}{{>{\dps}c}}
		\a_- + \a_+\wh{x}y + (\sum_{i=0}^k \a_i)\wh{x}y^{-1} = 0 \\
		(\sum_{i=0}^k \a_i) + \a_-\wh{x}y + \a_+\wh{x}y^{-1} = 0 \\
		\a_+ + \a_-\wh{x}y^{-1} + (\sum_{i=0}^k \a_i)\wh{x}y = 0 \\
	\end{array}
	\end{equation*}
	\par
	
		Now again by using \Cref{x hat y 39} we can conclude that:
	\begin{equation}
	\label{zero 39}
	\begin{array}{*3{>{\dps}l}}
		\a_-				+ \e(\a_+ + \sum_{i=0}^k \a_i) \wh{x}y^{-1} & = & 0 \\
		(\sum_{i=0}^k \a_i) + \e(\a_- +				 \a_+) \wh{x}		& = & 0 \\
		\a_+				+ \e(\a_- + \sum_{i=0}^k \a_i) \wh{x}y		& = & 0 \\
	\end{array}
	\end{equation}
	
		As mentioned above $\a = \sum_i \a_i$ where $\sup(\a_i) \subseteq \c_i$
	and by definition of $\c_i$'s from \Cref{conjugacy 39}, we can write:
	\begin{equation*}
	\begin{array}{*3{>{\dps}l}}
		\sum_{i=0}^k \a_i & = &	\sum_{i=0}^{3k} a_i	  x^i		 \\
		\a_+			  & = & \sum_{i=0}^{3k} a^+_i x^i y		 \\
		\a_-			  & = & \sum_{i=0}^{3k} a^-_i x^i y^{-1} \\
	\end{array}
	\end{equation*}
	\par
	
		By substitute of each $\a_i$'s in \Cref{zero 39},
	we can calculate the coefficients of each element of the group in the left sides of equations
	and since right sides are zero, so each coefficient must be zero too.
	Therefore, for every $h$, $i$ and $j$ we have:
	\begin{equation*}
	\begin{array}{*5{>{\dps}c}}
		a^-_h + \e(\a_+ + \sum_{r=0}^k \a_r) = 0	   & \qquad		&
		a_i	  + \e(\a_- + \a_+)				 = 0	   & \qquad		&
		a^+_j + \e(\a_- + \sum_{r=0}^k \a_r) = 0								\\
		
		a^-_h = - \sum_{r=0}^{3k} (a^+_r + a_r	)	   & \qquad		&
		a_i	  = - \sum_{r=0}^{3k} (a^-_r + a^+_r)	   & \qquad		&
		a^+_j = - \sum_{r=0}^{3k} (a^-_r + a_r  )								\\
		
		a^-_0 = \dt = a^-_{3k}						   & \stack{\cw{|}}{r}{k} &
		a_0	  = \dt = a_{3k}						   & \stack{\cw{|}}{r}{k} &
		a^+_0 = \dt = a^+_{3k}													\\
	\end{array}
	\end{equation*}
	\par
	
		So by knowing $a_0^-$, $a_0$ and $a_0^+$, all coefficients can be computed.
	Also since $F$ is a field of characteristic 3, we have $3k+1=1$,
	therefore, we have $a_0^- + a_0 + a_0^+ = 0$, thus:
		$$\anh(s) = \{ a^-_0 \wh{x}y^{-1} + a_0 \wh{x} + a^+_0 \wh{x}y \ | \ a^-_0+a_0+a^+_0=0 \}$$
\end{proof}

		Let $s$ be as was in \Cref{g hat 39}, that is $s = \wh{S}_3$, then we have:

\begin{Prs}
	\label{nil 39}
		$\anh(s)$ is a nilpotent ideal.
\end{Prs}

\begin{proof}
		Let $\a , \b , \g \in \anh(s)$, according to \Cref{g hat 39}:
	\begin{equation}
	\begin{alignedat}{5}
		& \a && = a^- && \wh{x}y^{-1} + a && \wh{x} + a^+ && \wh{x}y \\
		& \b && = b^- && \wh{x}y^{-1} + b && \wh{x} + b^+ && \wh{x}y \\
		& \g && = c^- && \wh{x}y^{-1} + c && \wh{x} + c^+ && \wh{x}y \\
	\end{alignedat}
	\end{equation}
	\par
	
		So we have:
	\begin{equation}
	\begin{split}
		\a . \b . \g & = (a^- \wh{x}y^{-1} + a \wh{x} + a^+ \wh{x}y)
					   . (b^- \wh{x}y^{-1} + b \wh{x} + b^+ \wh{x}y)
			   		   . (c^- \wh{x}y^{-1} + c \wh{x} + c^+ \wh{x}y) \\
					 & = (a^+-a^-)(b^+-b^-)\wh{G}.(c^- \wh{x}y^{-1} + c \wh{x} + c^+ \wh{x}y) \\
					 & = (a^+-a^-)(b^+-b^-)(c^-+c+c^+)\wh{G}|\langle x \rangle|
	\end{split}
	\end{equation}
	\par
	
		By \Cref{g hat 39}, $\a . \b . \g = 0$, thus $\anh^3(s) = 0$,
	therefore, $\anh(s)$ is a nilpotent ideal. \\
\end{proof}

		Let $s$ be as was in \Cref{nil 39}, that is $s = \wh{S}_3$, then we have:

\begin{Prs}
	\label{anh inclusion 39}
	$\anh(s) \subseteq J(FG)$.
\end{Prs}

\begin{proof}
		Since every nilpotent ideal is a nil ideal, so \Cref{nil 39} shows $\anh(s)$ is a nil ideal.
	On the other hand, by \cite[Lemma 2.7.13 on p.~109]{Milies2002},
	Jacobson radical contains all of the nil ideals, so:
		$$\anh(s) \subseteq J(FG)$$
\end{proof}

In the next corollary, we'll show that the equality hold:

\begin{Crl}
	\label{anh equal 39}
		$J(FG) = \anh(s)$.
\end{Crl}

\begin{proof}
		By \Cref{anh inclusion 39}, $\anh(s) \subseteq J(FG)$ and we know
	from \Cref{lemma} part (3) that	${J(FG) \subseteq \anh(s)}$, so the equality is hold:
		$$J(FG) = \anh(s)$$
\end{proof}

		We'll need the following proposition in the next steps:

\begin{Prs}
	\label{dim jacobson 39}
		$\dim(J(FG)) = \dim(\anh(s)) = 2$
\end{Prs}

\begin{proof}
		By \Cref{g hat 39} and \Cref{anh equal 39} we have:
	\begin{equation}
	\label{gen jacobson 39}
		J(FG) = \anh(s) = \{ a^-_0 \wh{x}y^{-1} + a_0 \wh{x} + a^+_0 \wh{x}y \ | \ a^-_0+a_0+a^+_0=0 \}
	\end{equation}

		That means, $J(FG)$ and $\anh(s)$ are generated by three elements with one restriction, hence:
			$$\dim(J(FG)) = \dim(\anh(s)) = 3 - 1 = 2$$
\end{proof}

		Let $H \cq \langle x \rangle \unlhd \ G$, a normal subgroup of $G$.
	Also we recall augmentation ideals $\d(G,H) = \langle h-1 | \ h \in H \rangle$,
	that in special case $H=G$,	we denote $\d(G) = \d(G,G)$.
	Now it's obvious that, by using \cite[Proposition 3.3.3 on p.~135]{Milies2002}, we have:
	\begin{equation*}
	\begin{split}
		\dim (\d(G,H)) = |G| - [G:H] \\
		\dim (\d(G,G)) = |G| - [G:G]\\
	\end{split}
	\end{equation*}

		Therefore, we can bring the following remark:

\begin{Rmk}
	\label{dim delta 39}
		Dimensions of $\d(G,H)$ and $\d(G)$ can be computed as follows:
	\begin{equation*}
	\begin{split}
		\dim (\d(G,H)) = 3m-3 \\
		\dim (\d(G,G)) = 3m-1 \\
	\end{split}
	\end{equation*}
\end{Rmk}

		We want to represent a decomposition for $\d(G)$ over $J(FG)$ and $\d(G,H)$.
	As both of them are included in $\d(G)$, first we show they are disjoint:

\begin{Prs}
	\label{intersection 39}
	$J(FG) \cap \d(G,H) = 0$.
\end{Prs}

\begin{proof}
		Let $\a \in J(FG) \cap \d(G,H)$. By \Cref{gen jacobson 39}, $J(FG) = \langle \wh{G} \rangle$.
	Now we compute $\a . \wh{x}$ in two different ways,
	according to see $\a$ as an element of $J(FG)$ or $\d(G,H)$ separately.
	Before that notice we know $|\langle x \rangle| = Ord_G(x) = m = 1$. So we have:
	\begin{equation*}
	\begin{array}{*3{>{\dps}c}}
		\a \in J(FG) = \langle \wh{G} \rangle
			& \qquad \qquad & \a \in \d(G,H) = \langle x-1 \rangle                 \\
		\a = a. \wh{G}
			& \qquad \qquad & \a = \b (x-1)                                        \\
		\a \wh{x} = a \wh{G} \wh{x} = a \wh{G} |\langle x \rangle|
			& \qquad \qquad & \a \wh{x}= \b (x-1) \wh{x}= \b (x \wh{x} - 1 \wh{x}) \\
		= a. \wh{G} . m = a . \wh{G} = \a
			& \qquad \qquad & = \b . (\wh{x} - \wh{x})= \b . 0 = 0                 \\
	\end{array}
	\end{equation*}
	\smallskip
	
	Therefore, we conclude that:
	\begin{equation}
	\label{alpha 39}
		\a = \a . \wh{x} = 0
	\end{equation}
	
	And hence we have:
	$$J(FG) \cap \d(G,H) = 0$$
\end{proof}

	Now the decomposition can be achieved:

\begin{Prs}
	\label{decomposition 39}
	$\d(G) = J(FG) \oplus \d(G,H)$.
\end{Prs}

\begin{proof}
		By \Cref{dim jacobson 39} and \Cref{dim delta 39}, we have:
			$$\dim (J(FG)) + \dim (\d(G,H)) = 2 + (3m-3) = (3m-1) = \dim (\d(G))$$
		Now \Cref{intersection 39} together with above equality shows that:
			$$\d(G) = J(FG) \oplus \d(G,H)$$
\end{proof}

		In the next Proposition, we prove that $\d(G,H)$ is a semisimple ring:

\begin{Prs}
\label{semisimple 39}
	$\d(G,H)$ is a semisimple ring.
\end{Prs}

\begin{proof}
		By \Cref{decomposition 39}, we have $\d(G,H) = \d(G)/J(FG) \subseteq FG/J(FG)$.
	From \cite[Theorem 6.6.1 on p.~214]{Milies2002}
	group ring of a field over a finite group is Artinian,
	so $FG$ is an Artinian ring, and \cite[Lemma 2.4.9 on p.~87]{Milies2002},
	implies its quotient ring, $FG/J(FG)$, is an Artinian ring too.
	Also from \cite[Lemma 2.7.5 on p.~107]{Milies2002} we know that ${J(FG/J(FG)) = 0}$. Now by using
	\cite[Theorem 2.7.16 on p.~111]{Milies2002} we can explore that $FG/J(FG)$ is semisimple,
	and by \cite[Proposition 2.5.2 on p.~91]{Milies2002}, all of its subrings are semisimple too.
	So $\d(G,H)$ is semisimple. \\
\end{proof}

		By Artin-Wedderburn Theorem, $\d(G,H)$ decomposes to its simple components
	that are division rings of matrices over extensions of $F$.
	Now we need to know their numbers and dimensions.
	First we show that the center of $\d(G,H)$ is included in the center of $FG$:

\begin{Prs}
\label{center inclusion 39}
	$Z(\d(G,H)) \subseteq Z(FG)$
\end{Prs}

\begin{proof}
		For the proof of this proposition, we need show that each element of $Z(FG)$
	must commute with all of elements of $FG$.
	Since $F$ is commutative and $G$ is generated by $x$ and $y$,
	so it suffices to show they commute with $x$ and $y$. Let $\a \in Z(\d(G,H))$,
	so it commutes with $x-1$ as it is in $\d(G,H)$:
	\begin{equation*}
	\begin{array}{*1{>{\dps}c}}
		\a . (x-1)   = (x-1) . \a      \\
		\a .  x - \a =  x    . \a - \a \\
		\a .  x      =  x    . \a      \\
	\end{array}
	\end{equation*}
	\par
	
		So $\a$ commutes with $x$. Now we show that $\a$ also commutes with $y$.
	First we show that $\a y - y \a$ is in $\anh(x-1)$. Notice we know that
	$(x-1)y = y(x^{-1}-1) \in \d(G,H)$, so:
	\begin{equation*}
	\begin{array}{*3{>{\dps}c}}
		(x-1)y \in \d(G,H)             & \qquad \qquad &  y(x-1) \in \d(G,H)             \\
  		\a . (x-1) . y = (x-1) . y .\a & \qquad \qquad &  \a . y . (x-1) = y. (x-1) . \a \\
		(x-1) . \a y = (x-1) . y \a    & \qquad \qquad &  \a y . (x-1) = y \a . (x-1)    \\
		(x-1) (\a y - y \a) = 0        & \qquad \qquad &  (\a y - y \a) (x-1) = 0        \\
	\end{array}
	\end{equation*}
	\smallskip
	
		So $(\a y - y \a) \in \anh (x-1)$ and by \cite[Lemma 3.4.3 on p.~139]{Milies2002}
	we know that ${\anh (x-1) = \anh(\d(G,H)) = FG \wh{x}}$.
	Now we compute $(\a y - y \a) . \wh{x}$ in two different ways,
	directly itself or see $(\a y - y \a)$ as an element of $FG . \wh{x}$ separately.
	Before that notice $\a \in Z(\d(G,H)) \subseteq \d(G,H)$, so by \Cref{alpha 39} $\a . \wh{x} = 0$.
	So we have:
	\begin{equation*}
	\begin{array}{*3{>{\dps}c}}
		(\a y - y \a) . \wh{x} = \a . y . \wh{x} - y . \a . \wh{x} = & \qquad \qquad &
			(\a y - y \a) . \wh{x} = \b . \wh{x} . \wh{x} =				  \\
		\a \wh{x} . y - y . \a \wh{x} = 0 . y - y . 0 = 0			 & \qquad \qquad &
			\b . \wh{x} . |\langle x \rangle| = \b \wh{x} = (\a y - y \a) \\
	\end{array}
	\end{equation*}
	\par
	
		Hence $\a y - y \a = (\a y - y \a) . \wh{x} = 0$, thus $\a y = y \a$,
	that means $\a$ also commutes with $y$ and therefore:
			$$Z(\d(G,H)) \subseteq Z(FG)$$
\end{proof}

	In the next proposition, we obtain the exact structure of $Z(\d(G,H))$:

\begin{Prs}
\label{center 39}
	$Z(\d(G,H)) = \langle \wh{\c}_1 , \dt , \wh{\c}_k \rangle$
\end{Prs}

\begin{proof}
		Let $\a\in Z(\d(G,H))$, from \cite[Theorem 3.6.2 on p.~151]{Milies2002} we know that
	${Z(FG)=\langle \wh{\c}_i \ | \ \forall i \rangle}$,
	so for center of augmentation ideal we have
	${Z(\d(G,H)) \subseteq \langle \wh{\c}_i \ | \ \forall i \rangle}$,
	by using \Cref{center inclusion 39}. So $\a = \sum_i r_i \wh{\c}_i$.
	By \Cref{alpha 39}, $\a . \wh{x} = 0$ and notice that
	$x^i \wh{x} = \wh{x}$, so ${(x^{i/t} + x^i + x^{it}) \wh{x} = 3 \wh{x} = 0}$. Hence:
	\begin{equation}
	\label{alpha x}
	\begin{split}	
		0 & = \a \wh{x} = \sum_i r_i \wh{\c}_i \wh{x} = r_0 \wh{\c}_0 \wh{x}
			+ (\sum_{i=1}^k r_i \wh{\c}_i \wh{x}) + r_- \wh{\c}_- \wh{x} + r_+ \wh{\c}_+ \wh{x} \\
	  	  & = {r_0 . 1 . \wh{x} + (\sum_{i=1}^k r_i (x^{i/t} + x^i + x^{it}) \wh{x})
			+  r_- \wh{x}y^{-1} \wh{x} + r_+ \wh{x}y \wh{x}}									\\
		  & =  r_0 \wh{x} + 0	+ r_- \wh{x}y^{-1} + r_+ \wh{x}y \cw{\sum_{i=1}^k}			\\
		  & =  r_- \wh{x}y^{-1} + r_0 \wh{x}	   + r_+ \wh{x}y \cw{\sum_{i=1}^k}			\\
	\end{split}
	\end{equation}
	\par
	
		Since left side of \Cref{alpha x} is zero, so right side coefficients should be zero too,
	so ${r_{-3} = r_0 = r_{+3} = 0}$, hence ${\a = \sum_{i=1}^k r_i \wh{\c}_i}$, that means
	${Z(\d(G,H)) \subseteq \langle \wh{\c}_1 , \dt , \wh{\c}_k \rangle}$.
	Now it suffices to show that all of these types of elements are included in $\d(G,H)$.
	We must show that there is a $\b$ that $\a = \b (x-1)$.
	It is straightforward to find $\b$'s coefficients by solving a system of linear equations.
	So $\a \in \d(G,H)$, and therefore:
			$$Z(\d(G,H)) = \langle \wh{\c}_1 , \dt , \wh{\c}_k \rangle$$
\end{proof}
\vspace{31pt}

Now the dimension of the center of $\d(G,H)$ can be computed:

\begin{Crl}
\label{dim center 39}
	$\dim ( Z( \d(G,H) ) ) = k$
\end{Crl}

\begin{proof}
		By \Cref{center 39},we know that $Z(\d(G,H)) = \langle \wh{\c}_1 , \dt , \wh{\c}_k \rangle$, so:
			$$\dim ( Z( \d(G,H) ) ) = k$$
\end{proof}

		We first solve the problem for group $G=T_{39}$ in the following section.

\section{\bf{\larger[1]Unit Group of $\f_{3^n} T_{39}$}}\label{sec4}

		Let $T_{39} = \langle x,y\,|\,x^{13}=y^3=1,\,x^y=x^3\rangle$,
	$C_n$ be cyclic group of order $n$, $M_n(R)$ be the ring of the square matrices
	of the order $n$ on the ring $R$ and $GL_n(R)$ be its unit group.
	Also $R^n$ be the direct sum of $n$ copy of the ring $R$, \ie $R^n = \op_{i=1}^n R$
	and $F_n$ be the extension of the finite field $F$ of the order $n$ that is $[F_n:F] = n$.
	So we have:

\begin{The}
	\label{main 39}
		Let $G  = T_{39}$ and $F = \f_{3^n}$.
	Then the structure of $\u(FG)$ can be obtained as follows: \medskip
			$$\u(FG)=C_3^{2n} \ti C_{3^n-1} \ti GL_3(F)^4$$
\end{The}
\par

\begin{proof}
		Let $\a \in Z(\d(G,H))$. From \Cref{center 39}, we know that $\a$ can be written as
	${\a = r_1\wh{\c}_1+r_2\wh{\c}_2+r_3\wh{\c}_3+r_4\wh{\c}_4}$.
	Since $char(F) = 3$, we have:
	\begin{equation*}
	\begin{alignedat}{5}
		& \a	  && = r_1		\wh{\c}_1	&& + r_2	  \wh{\c}_2
				  && + r_3		\wh{\c}_3	&& + r_4	  \wh{\c}_4   \\
		& \a^3	  && = r_1^3	\wh{\c}_1^3 && + r_2^3	  \wh{\c}_2^3
				  && + r_3^3	\wh{\c}_3^3 && + r_4^3	  \wh{\c}_4^3 \\
		& \a^3	  && = r_1^3	\wh{\c}_1	&& + r_2^3	  \wh{\c}_2
				  && + r_3^3	\wh{\c}_3	&& + r_4^3	  \wh{\c}_4   \\
		& \a^{3n} && = r_1^{3n} \wh{\c}_1	&& + r_2^{3n} \wh{\c}_2
				  && + r_3^{3n} \wh{\c}_3	&& + r_4^{3n} \wh{\c}_4   \\
		& \a^{3n} && = r_1		\wh{\c}_1	&& + r_2	  \wh{\c}_2
				  && + r_3		\wh{\c}_3	&& + r_4	  \wh{\c}_4   \\
	\end{alignedat}
	\end{equation*}
	\par
	
		Since $|F| = 3^n$, we know $r_i^{3^n} = r_i$, so $\a^{3^n} = \a$. Therefore, we have:
	$$\begin{array}{*2l}
		\d(G,H) \cong M_3(F)^4 \\
	  \end{array}$$
	\par
	
		By \cite[Proposition 3.6.7 on p.~153]{Milies2002}, $FG \cong F(G/H) \oplus \d(G,H)$,
	so for unit group $\u (FG) \cong \u (F(C_3)) \times \u (\d(G,H))$.
	Therefore, we have:
		$$\u (FG) = C_3^n \ti C_{3^n-1} \ti GL_3(F)^4$$
\end{proof}

\section{\bf{\larger[1]Unit Group of $\f_{3^n} T_{3m}$}}\label{sec5}
		For generalization, it suffices to do the same, as has been done in the previous section,
	for given $m$. In this regard, we write a \scalerel*{\includegraphics{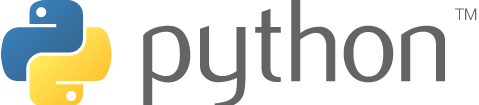}}{Qq} code,
	which does it for any $m$. {\href{https://github.com/A30hK0ch0l0/Unit-Group-of-F_3q-T_3n}
	{This code is available here in} \scalerel*{\includegraphics{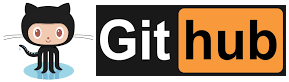}}{Qq},
	that takes $k$ and $t$ and returns unit group just like \Cref{main 39}.
	Everyone without Python 3 compiler on their computer, can executes it online.
	For example, \href{https://www.tutorialspoint.com/execute_python3_online.php}
	{here is an online Python 3 compiler}. \\

{\small

\bibliographystyle{spmpsci}		
\bibliography{References}		

\bigskip

{\small
{\em Author' addresse}:
{\em Ali Ashja'}, \href{https://www.modares.ac.ir/en}{Tarbiat Modares University, Tehran, Iran.} \par
e-mail: \texttt{\href{mailto:Ali.Ashja@Modares.ac.ir}{Ali.Ashja@\allowbreak Modares.ac.ir},
				\href{mailto:AliAshja@Gmail.com}	 {AliAshja@\allowbreak Gmail.com}	 }.
}

\end{document}